\documentclass[12pt]{article}

\usepackage{cmap}  
\usepackage[T2A]{fontenc}
\usepackage[utf8]{inputenc}
\usepackage[english]{babel}
\usepackage[usenames]{color}
\usepackage{amsmath,amssymb,amsthm}
\usepackage{arcs}
\usepackage{epsfig}
\usepackage{filecontents}
\usepackage{graphicx}
\usepackage[pdftex,colorlinks=true,linkcolor=blue,urlcolor=red,unicode=true,hyperfootnotes=false,bookmarksnumbered]{hyperref}
\usepackage{ifthen}
\usepackage{import}
\usepackage{indentfirst}
\usepackage{mathtools}
\usepackage{cancel}
\usepackage{xcolor}

\newcommand{\ff}{\mathcal{F}}

\newcommand{\Prb}{\mathsf{P}}

\newtheorem{thm}{Theorem}
\newtheorem{prob}[thm]{Problem}
\newtheorem{cla}[thm]{Claim}
\newtheorem{lemma}[thm]{Lemma}
\newtheorem{cor}[thm]{Corollary}
\newtheorem{prop}[thm]{Proposition}

\title{VC-saturated set systems}
\author{N\'ora Frankl$^{1,2}$ \and Sergei Kiselev$^{1}$ \and Andrei Kupavskii$^{1,3,4}$ \and Bal\'azs Patk\'os$^{1,5}$ \\
\small $^1$ Laboratory of Combinatorial and Geometric Structures,\\ \small Moscow Institute of Physics and Technology\\
\small $^2$ Carnegie Mellon University, Pittsburgh\\
\small $^3$ Institute for Advanced Study, Princeton\\
\small $^4$ G-SCOP, CNRS, Grenoble\\
\small $^5$ Alfr\'ed R\'enyi Institute of Mathematics, Budapest}
\date{}

\begin{document}

\maketitle

\begin{abstract}
    The well-known Sauer lemma states that a family $\ff\subseteq 2^{[n]}$ of VC-dimension at most $d$ has size at most  $\sum_{i=0}^d\binom{n}{i}$. We obtain both random and explicit constructions to prove that the corresponding saturation number, i.e., the size of the smallest maximal family with VC-dimension $d\ge 2$, is at most $4^{d+1}$, and thus is independent of $n$.
\end{abstract}
\section{Introduction}
In this paper, we consider a set theoretic problem concerning the Vapnik-Chervonenkis dimension of set families. This notion plays a central role in statistical learning theory \cite{BluEhrHau89,VapChe68}, discrete and computational geometry \cite{Mat13} and several other areas of mathematics \cite{FurPac91,KreNisRon99}.

For a family $\ff$ and a set $X,$ let $\ff|_X:=\{F\cap X: F\in \ff\}$ be the {\it projection} of $\ff$ onto $X$ (also called the \textit{trace} of $\ff$ on $X$). We say that $\ff$ \textit{shatters} $X$ if $\ff|_X=2^X$. The {\it VC-dimension} of $\ff$, denoted $VC(\ff)$ is the size of the largest $X$ shattered by $\ff$. The following seminal result, often called the Sauer lemma, relates the size of a set family with its VC-dimension. 

\begin{thm}[Sauer \cite{sau72}, Shelah \cite{She72}, Vapnik, Chervonenkis \cite{VapChe68}]\label{vc}
If $\ff \subseteq 2^{[n]}$ has VC-dimension at most $d$, then $|\ff|\le \sum_{i=0}^d\binom{n}{i}$.
\end{thm}

The bound of Theorem \ref{vc} is sharp as shown by all subsets of $[n]$ of size at most $d$, but there are many other extremal families achieving this size (see e.g. \cite{FurQui84}). Alon \cite{Alo83} and Frankl \cite{Fra83} reproved Theorem \ref{vc} independently, both introducing a down-shifting technique that is often used to address extremal problems on traces of finite sets. For one more possible proof of Theorem \ref{vc}, see \cite{FraPac83}. Pajor \cite{Paj85} strengthened Theorem \ref{vc} to the inequality $|\ff|\le |Sh(\ff)|$ where $Sh(\ff)$ stands for the family of all sets shattered by $\ff$ (a dual inequality was obtained by Bollob\'as, Leader, and Radcliffe \cite{BolLeaRad89} and then by Bollob\'as and Radcliffe \cite{BolRad95}). Examining families satisfying these inequalities with equality has been studied lately (see e.g. \cite{MesRon13}). For more extremal set theoretic problems on traces of set families, see Chapter 8 of \cite{GerPat18}.

\smallskip

We study the {\it saturation} problem for families with fixed VC-dimension. We say that $\ff\subset 2^{[n]}$ is {\it saturated} if $VC(\ff)<VC(\ff')$ for every $\ff'\subset 2^{[n]}$ such that $\ff'\supsetneq \ff$, and $\ff$ is $d$-saturated if it is saturated and $VC(\ff)=d$. Answering a question of Frankl \cite{Fra89}, after work by Alon, Moran, and Yehudayoff \cite{AloMorYeh16}, Balogh, M\'esz\'aros, and Wagner determined \cite{BalMesWag18} the asymptotics of the logarithm of the number of $d$-saturated families $\ff\subseteq 2^{[n]}$. We will be interested in the \textit{saturation number} $sat_{VC}(n,d)$, the minimum size of a $d$-saturated family $\ff\subseteq 2^{[n]}$. Clearly, $sat_{VC}(n,0)=1$ for any $n$, as any set forms a $0$-saturated family. Dudley showed \cite{Dud85} (see also \cite{Dudl99}) that $sat_{VC}(n,1)=n+1$ for all values of $n$. Together with Theorem \ref{vc}, this implies that any $1$-saturated family $\ff\subseteq 2^{[n]}$ has size $n+1$. Our main result shows that for larger values of $d$, the situation is completely different.

\begin{thm}\label{thmmain}
For any $d\geq 3$, 
$sat_{VC}(n,d-1)\le 4^d$ 
holds for any $n\ge 2d$. Moreover, if $d$ is odd or if $d\ge 14$, then we can replace $4^d$ with $\frac{1}{2}\binom{2d}{d}$.
\end{thm}


One interesting question that remains is to find better lower bounds on $sat_{VC}(n,d-1).$ As we show above, it is at most roughly $4^d$ for most $d$. On the other hand, a trivial lower bound is $sat_{VC}(n,d-1)\ge 2^d-1$ because at least $1$ $d$-set should contain $2^{d}-1$ projections. It is not difficult to get a slightly better bound $sat_{VC}(n,d-1)\ge 2^d$, but a more significant improvement remains elusive. 

\begin{prob} Show that $sat_{VC}(n,d-1)\ge c^d$ for some $c>2.$
\end{prob}

\section{Searching for the upper bound constructions}

We turn to the proof of our main result. Observe that in order to obtain an upper bound on the saturation number, one needs constructions. The following proposition gives us an idea on how constant-sized saturated families should look like. In order to formulate it, we need some definitions. For a family $\ff\subset 2^{[n]}$ and $x,y\in [n]$, we say that $x$ and $y$ are {\it duplicates}, if, for any $F\in \ff,$ $x\in F$ if and only if $y\in F$. Let $D(x)\subset [n]$ be the class of all duplicates of $x$ with $x$ included. 
Define the {\it reduced family} $\mathcal R(\ff)$ to be the projection of $\ff$ on $W,$ where $W\subset [n]$ is obtained by keeping exactly one element out of each class of duplicates.  
Note that $\mathcal R(\ff)$ is defined up to relabeling of the ground set, $|\mathcal R(\ff)| = |\ff|$ and, informally, $\mathcal R(\ff)$ captures the structure of $\ff.$ In the next proposition $\Delta$ denotes the symmetric difference of sets.

\begin{prop}\label{equiv} Let $d\ge 2$ and consider a $d$-saturated family $\ff\subset 2^{[n]}$.

(i) Assume that $\ff = \mathcal R(\ff)$. If $x_1,\ldots, x_m\in [n]$ are such that, for any $F\in \ff$ and $x_i$, the set $F\triangle \{x_i\}$ is not contained in $\ff$, then a family $\ff'$ (on a larger ground set) that is obtained from $\ff$ by duplicating some of $x_1,\ldots, x_m$ is $d$-saturated.

(ii) If there exists $x\in[n]$ such that $|D(x)|\ge 2$, then for any such $x$ 
the family $\mathcal R(\ff)$ must satisfy the property from (i) w.r.t. (a duplicate of) $x$. That is, for any $F\in \ff,$ the set $F\triangle D(x)$ is not contained in $\ff$.
\end{prop}
We note that the condition of (ii) definitely holds for some $x$ if $n>2^{|\ff|}$. This proposition implies that a constant-sized $d$-saturated family for any sufficiently large $n$ is reducible to a saturated family as in (i).

\begin{proof} (i) For simplicity, assume that $\ff'$ is obtained from $\ff$ by duplicating $x_1$ several times, and let $D(x_1)$ be the class of duplicates of $x_1$. Assume that $\ff'$ is not saturated, that is, there is a set $X\notin \ff'$ such that $\ff'\cup \{X\}$ has $VC$-dimension $d.$ Recall that $\ff'|_{[n]} = \ff,$ and $\ff$ is saturated. Thus, $X|_{[n]} = F|_{[n]}$ for some $F\in \ff'$. In other words, $\emptyset \ne F\Delta X \subset D(x_1)\setminus \{x_1\}$. Take any $y\in  F\Delta X $, define $Y:=[n]\setminus \{x_1\}\cup \{y\}$ and consider $\ff_0:=\ff'|_{Y}.$ Then $\ff_0$ is isomorphic to $\ff.$ By the choice of $y$, $F|_Y\Delta X|_Y = \{y\}$, and thus, by the definition of $x_1$ (and $y$ being the duplicate of $x_1$ for $\ff'$), only at most one of $F|_Y$ and $X|_Y$ can be contained in $\ff_0$. Therefore, $X|_Y\notin \ff_0$, and thus $VC\big(\ff_0\cup \{X|_Y\}\big)>VC(\ff_0),$ a contradiction.

(ii) The proof of this part is largely the proof of (i) in reverse. Assume that this is not the case. Take $F,F\Delta D(x) \in \ff$, put $Y:=(F\setminus D(x))\cup \{x\}$  and consider the family  $\ff_1:=\ff\cup\{Y\}.$ Clearly, $|\ff_1|=|\ff|+1$. Next, we show that $VC(\ff_1) =VC(\ff)=d,$ contradicting the saturation property of $\ff.$ Indeed, assume that some $(d+1)$-element set $S$ is shattered by $\ff_1.$ Then, clearly, $S\cap D(x)\ne \emptyset$. Moreover, if $|S\cap D(x)|=1$ then $Y|_{S}\in \big\{F|_{S}, F\Delta D(x)|_{S}\big\},$ and thus such $S$ should have been shattered by $\ff.$ Therefore, $|S\cap D(x)|\ge 2.$ However, by definition, there is at most $1$ set in $\ff_1$ that does not either contain or is disjoint with $S\cap D(x),$ while, in order to shatter $S,$ one needs at least $2^d$ such sets. This contradiction shows that $VC(\ff_1)=d$, and thus $\ff$ was not saturated in the first place. 
\end{proof}

One of the challenges in proving Theorem~\ref{thmmain} was to find the right class of families to search for constructions in. Proposition~\ref{equiv} suggests to search for (reduced) saturated families such that the Hamming distance between any two sets in the family is at least $2$. One natural way to achieve this is to consider {\it uniform families}, i.e., families in which all sets have the same size. Let us denote the set of all $k$-element subsets of $[n]$ by ${[n]\choose k}$.

It turned out that we can find $(d-1)$-saturated families among intersecting families in ${[2d]\choose d}$. The following proposition gives us a sufficient condition for such a family to be $(d-1)$-saturated. We say that $\ff$ \textit{almost shatters} $X$ if $\ff|_X=2^X\setminus \{\emptyset\}$ or $\ff|_X=2^X\setminus \{X\}$.

\begin{prop}\label{almost}
    If a family $\ff\subset {[2d]\choose d}$ almost shatters any  $A\in {[2d]\choose d}$, then $\mathcal \ff$ is $(d-1)$-saturated.
\end{prop}

\begin{proof}
    Since $\ff$ is almost shattered, adding a $d$-set to $\ff$ will result in shattering that set. We thus need to show that adding a set $B$ of size other than $d$ also results in some $d$-set being shattered. The argument is symmetric for sets of size smaller/larger than $d$, and we present the case $|B|<d$ only. Consider a family $\ff':=\ff\cup\{B\}$, $|B| < d$.  
    Take a set $X \subset \overline B$, $|X| = d - |B|$. Then $X \in \ff|_{B\cup X}$, and so by the assumption on $\ff$ there is a $d$-set $A\in\mathcal F$ such that $A\cap (B\cup X )= X$. Therefore, $B\cap A = \varnothing$ and thus $A$ is shattered by $\ff'$.
\end{proof}

Proposition~\ref{equiv} implies that any uniform $d$-saturated family on a ground set of size $n$ can be transformed into a $d$-saturated family of the same size on any larger ground set. Proposition~\ref{almost} tells us that it is sufficient to find a family $\ff\subset {[2d]\choose d}$ that almost shatters any $d$-subset of $[2d].$ The latter property implies that, for any $d$-set $S$, exactly one of $S,\bar S$ must be contained in $\ff$. In other words, $\ff\subset {[2d]\choose d}$ must be an intersecting family of size $\frac 12{2d\choose d}.$ 

In Section \ref{random}, we show that for $d\ge 14$, if we pick one set from each such complementary pair independently and uniformly at random, then with positive probability, we obtain a family that almost shatters every $d$-set. 

In Section \ref{explicit2}, we give explicit constructions of saturated families for any $d\ge 4$ that are based on intersecting families as above and have an additive combinatorics flavour. For odd $d$, we also obtain a certain classification result. \\ 

Before going on to constructions for general $d$, let us give a concrete example of a saturated family for $d=3$, which proves Theorem~\ref{thmmain} for that case, as well as gives an idea of what type of intersecting families we are going to use for explicit constructions. 

Let $\mathcal{F}\subset \binom{[6]}{3}$ be the family of all $3$-tuples in which the sum of the elements belongs to  $H=\{1,3,4\}$ mod $6$. Note that $\sum_{i=0}^5i = 3$ (mod $6$) and that $H\cap (3-H) = \emptyset$, where here and in what follows the operations are mod $6.$ This implies that $\ff$ contains exactly $1$ set out of each complementary pair of $3$-sets and that, in particular, $\ff$ is intersecting.

\begin{cla} Every $A\in \binom{[6]}{ 3}$ is almost shattered by $\mathcal{F}$.
\end{cla}

\begin{proof} To prove the claim, it is sufficient to show that, for any $S'\subset S\in {[6]\choose 3}$, $|S'| \in \{1,2\},$ there exists a set $F\in \ff$ such that $F\cap S = S'.$ Assume that the sum of the elements from $S$ is $x$ (mod $6$) and the sum of elements from $S'$ is $y$ (mod $6$). 

If $|S'| = 2$ then we need to find $z\in \bar S$ such that $y+z\in  H.$ If there is no such $z$ then $\{y+z: z\in \bar S\} = \{0,2,5\}$ and so $\{y+z: z\in S'\} \subset \{1,3,4\}$. But then, assuming $S' = \{z_1,z_2\}$, we have that the sum of the two elements $(y+z_1)+(y+z_2) = 3(z_1+z_2) = 0$ (mod $3$), but on the other hand, it must be one of the numbers in $\{1+3, 1+4,3+4\}$ (mod $6$), and none of those numbers is divisible by $3$. This contradiction implies that there must be $z$ with the desired property.

The case $|S'| = 1$ is similar. Put $\bar S =\{z_1,z_2,z_3\}$. Assuming that there is no pair $z_i,z_j\in \bar S$, $i\ne j,$ such that $y+z_i+z_j\in \{1,3,4\},$ we get that $y+\{z_1+z_2,z_1+z_3,z_2+z_3\}  =\{0,2,5\},$ which, passing to the complements and using that $\sum_{i=0}^5 i = 3\ ({\rm mod}\ 6),$ means that $y'+y''+\{z_1,z_2,z_3\} = \{1,3,4\},$ where $\{y,y',y''\} = S.$ But then $y'+y''+\{y,y',y''\} = \{0,2,5\},$ and, in particular, $3(y'+y'')\in \{0+2,2+5,0+5\}$, which is a contradiction. 
This concluded the proof of the claim.
\end{proof}

Equipped with this claim, we apply Propositions~\ref{almost}, concluding that $\ff$ is saturated. We then apply Proposition~\ref{equiv} (i) and duplicate arbitrary elements sufficiently many times to get a saturated family of VC-dimension $2$ for $n\ge 6$.




\section{Random construction}\label{random}

Consider a random family $\mathcal F \subset \binom{[2d]}{d}$, obtained in the following way: for each pair $A,\bar A$ of complementary $d$-element sets, we include one of them in $\mathcal F$ independently and uniformly at random. Let $H_A$ be an event that $A \in \mathcal F$.

For any $d$-set $A$ and set $X \subseteq A$ let  $Q_{A,X}$ stand for  the event that $X\notin \ff|_A$. This event happens if and only if for each pair of complementary $d$-sets $B,\bar B$ such that $A \cap B = X$, we added $\overline B$ to $\mathcal A$. In particular, $$\Prb[Q_{A,X}]=2^{-\binom{d}{|X|}}.$$

\begin{thm}\label{thmrandom} If $d\ge 14$, we have $\Prb\big[\bigcap_{A\in {[2d]\choose d},\varnothing\ne X\subset A}\bar Q_{A,X}\big]>0,$ i.e., with positive probability $\ff$ almost shatters every $A\in {[2d]\choose d}.$
\end{thm}
Equipped with this theorem, we can conclude the proof of Theorem~\ref{thmmain} as in the case of $d=3$, given in the previous section.

We shall use Lov\'asz Local Lemma to show the validity of Theorem~\ref{thmrandom}. 
\begin{lemma}[Lov\'asz Local Lemma] Let $B_1,\ldots, B_m$ be events in an arbitrary probability space. For each $i,$ let $S_i\subset [m]$ be such that $B_i$ is independent of the sigma-algebra generated by the events $\{B_j: j\notin S_i\cup\{i\}\}$. Assume that there are real numbers $x_1,\ldots, x_m$ such that $0\le x_i<1$ and $$\Prb[B_i]\le x_i\prod_{j\in S_i}(1-x_j).$$
Then with positive probability no event $B_i$ holds.
\end{lemma}

Whether or not $H_B$ holds only depends on those events $Q_{A,X}$ for which $A\cap B = X$. Thus, an event $H_B$ depends on $\binom{d}{k}^2$ events $Q_{A,X}$ with $|X| = k$. Therefore, an event $Q_{A,X}$ depends on $$d_{|X|,l} := \binom{d}{|X|} \binom{d}{l}^2$$ events $Q_{B,Y}$ with $|Y| = l$.

To apply LLL, we need to choose the coefficients $x_i$. We put $$x_{A,X} := p_{|X|} := 2^{-\max\{{d-1\choose |X|},{d-1\choose d-|X|}\}}.$$
The cases $|X|\le d/2$ and $|X|\ge d/2$ are symmetric, and thus, in what follows, we assume that $|X|\le d/2$. Then the maximum in the expression above is attained on the first binomial coefficient and $\Prb[Q_{A,X}]/p_{|X|} = 2^{-{d\choose |X|}+{d-1\choose |X|}} = 2^{-{d-1\choose |X|-1}}$. We need to show that for each $1 \le k \le d/2$ and $|X|=k$ we have
$$
\Prb[Q_{A,X}] \le p_{|X|} \prod_{l=1}^{d-1} (1 - p_{l})^{d_{k,l}} \ \ \ \ \ \Leftrightarrow \ \ \ \ \ \prod_{l=1}^{d-1} (1 - p_{l})^{d_{k,l}}\ge 2^{-{d-1\choose k-1}}.
$$
Recall that $d_{k,l} = {d\choose k}{d\choose \ell}^2$ and that ${d-1\choose k-1}/{d\choose k} = 
\frac k{d}$ and is minimized for $k = 1$. Thus, to verify the last displayed inequality, it is sufficient to show that 
\begin{equation}\label{eqlll}\prod_{l=1}^{d-1} (1 - p_{l})^{{d\choose l}^2}\ge 2^{-\frac 1d}.\end{equation}
For $d/2> l\ge 2$ and $d\ge 10$ $$\frac{{d\choose l}^22^{-{d-1\choose l}}}{{d\choose l-1}^22^{-{d-1\choose l-1}}}=\frac {(d-l)^2}{l^2}2^{-\frac {d-2l}{d-l}{d-1\choose l}}\le (d-1)^22^{-\frac{(d-1)(d-4)}2}<\frac 1{10},$$
and so we have $$\prod_{l=1}^{d-1} (1 - p_{l})^{{d\choose l}^2}\ge \prod_{l=1}^{d/2} (1 - p_{l})^{2{d\choose l}^2} \ge 1-2\sum_{l=1}^{d/2} {d\choose l}^2 2^{-{d-1\choose \ell}}\ge 1-3d^2 2^{1-d}.$$
The last expression is at least $1-\frac 2d$ for any $d$ such that $2^d\ge 12 d^3.$ The latter holds for $d\ge 16$. On the other hand, for $d\ge 10$ we have $2^{-1/d}<1-\frac 1{2d}$, and thus \eqref{eqlll} holds for $d\ge 16.$ By doing a more careful calculation, one can verify that \eqref{eqlll} holds for any $d\ge 14$.


\section{Explicit constructions}\label{explicit2}
For odd $d\ge 7$ we find explicit constructions
of intersecting families $\mathcal{F}\in \binom{[2d]}{d}$ which almost shatters any $A\in\binom{[2d]}{d}$. We then conclude the proof as in the case $d=3.$ The details of this are in Subsection~\ref{sec31}. 

For even $d\geq 6$ the explicit constructions we found are slightly different. They consist of a maximal intersecting family in $\binom{[2d]}{d}$ and a few other sets, which form a saturated (and not necessarily almost-shattering) family. 
Thus these constructions are not necessarily uniform, moreover they may contain two sets whose Hamming distance is one. Therefore, in order to use Proposition~\ref{equiv} (i) and extend  the construction to larger $n$, we cannot simply duplicate an arbitrary element. 
However, we will make sure to have a distinguished element, for which the condition of Proposition~\ref{equiv} (i) holds. The details are given in Subsection \ref{sec32}.

Finally, in Subsection~\ref{sec33}, we give two examples of saturated families for $d=4,5$. Those together with the example for $d=3$ in the introduction cover all values of $d\ge 3$, as stated in Theorem~\ref{thmmain}. 

\subsection{Odd $d$}\label{sec31}
Fix an integer $d$ and  consider a set $X\subset [2d]$ of size $d$. Define $$\ff(X):=\Big\{F\in {[2d]\choose d}: \sum_{i\in F} i \in X{\rm \ (mod\ }2d)\Big\}.$$

\begin{thm}\label{thmnew1}
Let $d = 2k+1$. Then $\ff(X)$ almost shatters every $S\in {[2d]\choose d}$ if and only if the following three conditions hold:
\begin{enumerate} \item $|X|=d$ and $X\cap (d-X) = \emptyset\ {\rm (mod\ } 2d)$;
\item $X$ contains both odd and even elements;
\item for every $u\in X,$ $\sum_{w\in X\setminus\{u\}} w\ne 0\ {\rm (mod\ }d).$
\end{enumerate}
\end{thm}

It is not difficult to  find residue classes that satisfy the three conditions from the theorem. 
We use the notation $[a,b]:=\{a,a+1,\ldots, b\}$. Then one example is 
$$X:= [1,k] \cup [2k+1, 3k+1]$$ 
for odd $k.$ Indeed, both 1. and 2. are straightforward to check. To see 3., we note that all elements of $X$ are $0,1,\ldots, k \ ({\rm mod\ }2k+1),$ while $$\sum_{x\in X} x= 2\cdot {k+1\choose 2}  = k(k+1) = k+1+\frac {k-1}2(2k+2) = \frac{3k+1}2\ ({\rm mod\ }2k+1).$$ Thus, condition 3. is satisfied.

Another example for odd $k$ is
$$X:= \{1,3\ldots, 2k-1,2k+1,2k+2,2k+4\ldots, 4k\}.$$

The previous examples give construction in the case $d = 4r+3$ for some positive integer $r$. In case $d =4r+1\ge 9$, we can take 
$$X:= \{0\}\cup A\cup (d+A)\ ({\rm mod}\ 2d), \ \text{where } A =  \big\{1,\ldots, 2r-t-1,2r-t,2r+1,2r+2\ldots,2r+t\big\}$$
for some appropriately chosen $t\ge 1.$ E.g., for $k=4$ we can take $t=1$, getting the set $\{1,2,3,5\}\ ({\rm mod\ }9)$. It is not difficult to check the first two conditions. As for the third condition, note that $2(\sum_{i=1}^{2r-t}i+\sum_{i=2r+1}^{2r+t}i) = r+t^2{\rm \ (mod\ }4r+1).$ By taking $t$ such that $2r+t<r+t^2<4r+1$, which is always possible for $r\ge 2,$ we make sure that the third condition is satisfied.

\vskip+0.2cm


Let us prove Theorem~\ref{thmnew1}.

Take any such $X$. In what follows, we treat $X$ as a set of residues modulo $2d,$ and the inclusions/equality between $X$ and other sets should be interpreted as those for sets of residues modulo $2d.$ Most sums are also taken modulo $2d,$ which should be clear from the context. The proof of the theorem consists of the following lemmas.
\begin{lemma}\label{fullempty}
For any $A\in {[2d]\choose d},$ $\ff(X)|_A$ contains exactly $1$ out of  $\emptyset, A$ if and only if the first condition from Theorem~\ref{thmnew1} holds.
\end{lemma}\label{lem111}
\begin{proof}
Note that $\sum_{i=1}^{2d} i = d\ ({\rm mod\ }2d)$. Thus, the first condition in the theorem is equivalent to saying that for any $B\in {[2d]\choose d},$ $B\in \ff(X)$ if and only if $\bar B\notin \ff(X)$. 
\end{proof}

We will need some lower bounds on a special instance of the generalized Erd\H os--Heilbronn problem (originally \cite{ErdHei}, for a recent survey see Chapter IV A.3 in \cite{Baj18}). In a group $G$, the restricted $s$-sumset $\sum\binom{A}{s}$ for some $A\subseteq G$ and integer $s\ge 2$ is the set of all different sums of $s$ distinct elements from $A$ (in the number theory literature, the notation $s\wedge A$ is used). 
We will be interested in the case $G=\mathbb{Z}_{2d}$ and $|A|=d$. We will use a special case of the following result from \cite{GerGriHam12}. A \textit{2-coset} is a coset of a subgroup with all non-zero elements having order 2, and an \textit{almost 2-coset} is a 2-coset with possibly one element removed.

\begin{thm}[Girard, Griffiths, Hamidoune \cite{GerGriHam12}] Let $A$ be a subset of the abelian group $G$ and let $2 \le s \le |A|-2$. Then
$|\sum {A\choose s}| \ge |A|$
unless $s \in \{2, |A| -2\}$ and $A$ is a 2-coset. Furthermore
$|\sum {A\choose s}| > |A|$
unless $A$ is a coset of a subgroup of $G$ or $s \in \{2, |A|-2\}$ and

(i) $A$ is an almost 2-coset, or

(ii) $|A| = 4$ and $A$ is the union of two cosets of a subgroup of order 2.
\end{thm}

\begin{cor}\label{addcom} Consider $A\subset \mathbb Z_{2d}$, $|A| = d\ge 5$, such that $A$ contains both odd and even elements. Then for each $2\le s\le d-2,$ we have $|\sum\binom{A}{s}|>d$.
\end{cor}


\begin{lemma} The family $\ff(X)|_A$ contains 
all the sets of size $s,$ $2\le s\le d-2$ for any $A\in {[2d]\choose d}$ if and only if condition 2 from Theorem~\ref{thmnew1} holds.
\end{lemma}
\begin{proof}
If $X$ contains only even elements then it is not difficult to see that the $\ff(X)|_A,$ where $A$ is the set of all odd elements, misses all projections of odd size. Similarly, if $X$ contains only odd elements then $\ff(X)|_A,$ where $A$ is the set of all even elements, misses all projections of even size. Thus, condition 2 from the theorem is necessary.

Conversely, take any such $A$ and a particular subset $A'\subset A$ of size $s$. We need to show that it is possible to complement it with $d-s$ elements in $\bar A$ so that the sum of all elements in the resulting set belongs to $X$ modulo $2d.$
Applying Corollary~\ref{addcom}, we have $\Big(\sum_{a\in A'}a+ \sum\binom{\bar A}{d-s}\Big)\cap X \ne \emptyset$ (mod $2d$) by the pigeon-hole principle for any $\bar A$ containing both odd and even elements. In case when $\bar A$ is the set of all even or all odd elements, $\sum\binom{\bar A}{d-s}$ contains either all even or all odd elements. In any case,  condition 2 from the theorem implies that the aforementioned intersection is non-empty as well. 
\end{proof}
\begin{lemma}\label{projection1}
 The family $\ff(X)|_A$ contains 
all the sets of size $1$ and $d-1$ for any $A\in {[2d]\choose d}$ if and only if condition 3 from Theorem~\ref{thmnew1} holds.
\end{lemma}
\begin{proof} 
If $u\in X$ is such that $\sum_{w\in X\setminus \{u\}} w= 0$ (mod $2d$) then $\ff(X)|_X$ does not contain $X\setminus \{u\}.$ Indeed, we need an element $y$ from $\bar X$ to complement $X\setminus \{u\}$ to obtain a set with sum in $X$. But then $y\in X\cap \bar X.$ If $u\in X$ is such that $\sum_{w\in X\setminus \{u\}} w= d$ (mod $2d$), then consider the sets $A = d+X$ and $A' =  A\setminus \{u+d\}$. We have $\sum_{a\in A'} a = (d-1)d+\sum_{x\in X\setminus \{u\}} x = d$ (mod $2d$), because $(d-1)d = 0$ (mod $2d$) (here we use the fact that $d$ is odd). We claim that there is no $F\in \ff(X)$ such that $F\cap A = A'.$ Indeed, if there is $b\in \bar A$ such that $\sum_{y\in A'\cup \{b\}}y\in X$ then $b\in X-\sum_{y\in A'} y = X-d = A$, a contradiction. Thus, condition 3 from the theorem is necessary.\\

Conversely, fix some $A$. Let us first show that, for any $A'\subset A$ of size $d-1$, there is $b\in B:=\bar A$ such that $A'\cup \{b\}\in \ff(X)$. 

Assume that this is not the case. Then the set $B+\sum_{a\in A'} a\subset Y:=\bar X$ (mod $2d$), and, given that $|B| = |Y| = d$, we have 
$$B+\sum_{a\in A'} a = Y.$$ 
Let us put $z:= \sum_{a\in A'} a$ (mod $2d$). Then, since $B\sqcup A = X\sqcup Y$, we have $A+z = X$. Put $X':=\{a'+z:a'\in A'\}$. We get that $z = \sum_{a\in A'} a = \sum_{x\in X'} x-z|X'|$. Rewriting this, we have $\sum_{x\in X'} x = d z,$ and thus $\sum_{x\in X'} x = 0\ ({\rm mod\ } d)$. This contradicts condition 3 from the theorem.


\vskip+0.1cm

The case when $A' = \{a\}$ is very similar. We show that for any $a$ there is $b\in B$ such that $\{a\}\cup B\setminus \{b\}\in \ff(X)$. 

Assume that this is not the case. Then 
$$\Big(a+\sum_{b\in B} b\Big) - B=Y.$$ 
Let us put $z:= a+\sum_{b\in B} b$. Then, since $B\sqcup A = X\sqcup Y$, we have $z-A = X$. Put $x':=z-a$ and note that $x'\in X$. 
We get that $z = a+\sum_{b\in B} b = z-x'+dz-\sum_{y\in Y} y$. Recall that (cf. Lemma~\ref{lem111}) $\sum_{y\in Y}y+\sum_{x\in X} x = d$. Using the last two equations, we again get $\sum_{x\in X\setminus \{x'\}} x = 0\ ({\rm mod\ } d)$. This contradicts condition 3 from the theorem.
\end{proof}
This concludes the proof of Theorem~\ref{thmnew1}.\\

\subsection{Even $d$}\label{sec32}
The first part of the argument in this case follows essentially the same steps as the argument in the case of odd $d$. Fix an even integer $d$ and consider a set $X\subset [2d]$ of size $d-1$. Define 
\[\mathcal{F}_1(X)=\Big\{F\in {[2d]\choose d}: 2d\in F \textrm{ and }\sum_{i\in F} i \in X\cup \left \{\frac{d}{2}\right \}{\rm \ (mod\ }2d) \Big\}.\]
\[\mathcal{F}_2(X)=\Big\{F\in {[2d]\choose d}: 2d\notin F \textrm{ and }\sum_{i\in F} i \in X\cup \left \{\frac{3d}{2}\right \}{\rm \ (mod\ }2d) \Big\}.\]

\begin{prop}\label{almostshattered} Let $d=2k$. For $X\subset [2d]\setminus \{\frac{d}{2},\frac{3d}{2}\}$ the family $\mathcal{F}:=\mathcal{F}_1(X)\cup \mathcal{F}_2(X)$ almost shatters every $S\in \binom{[2d]}{d}$ with $2d\in S$ if the following three conditions hold.
\begin{enumerate}
    \item $|X|=d-1$ and $X\cap (d-X)=\emptyset$ (mod $2d$);
    \item $X$ contains both odd and even elements. 
    \item For $X_1=X\cup \{\frac{d}{2}\}$ and for every $u\in X_1$ we have $\sum_{w\in X_1\setminus \{u\}}w\neq 0\ {\rm (mod\ }d)$. Equivalently, for $X_1=X\cup \{\frac{3d}{2}\}$ and for every $u\in X_1$ we have $\sum_{w\in X_1\setminus \{u\}}w\neq 0\ {\rm (mod\ }d)$. 
\end{enumerate}
\end{prop}

The following three claims imply the proposition.

\begin{cla}\label{fullempty2}For every $S\in \binom{[2d]}{d}$ exactly one of $S, \overline{S}$ belongs to $\mathcal{F}$.
\end{cla}

\begin{proof}
If $\sum_{i\in S}i\in X$ it follows from the first condition as in Lemma \ref{fullempty}. If $\sum_{i\in S}i=\frac{d}{2}$ (mod $2d$) or $\sum_{i\in S}i=\frac{3d}{2}$ (mod $2d$), then $\sum_{i\in S}i=\sum_{i\in \overline{S}}i$ (mod $2d$), thus $S\in \mathcal{F}$ and $\overline{S}\notin \mathcal{F}$.
\end{proof}

\begin{cla}\label{size2} For any $S\in \binom{[2d]}{d}$ with $2d\in S$, every subset of size $2\leq s \leq d-2$ of $S$ appears in $\mathcal{F}|_S$.
\end{cla}

\begin{proof} Let $S'\subset S$ be a subset of size $s$. We consider two cases.

\bf{Case 1: } \rm $2d\in S'$. In this case we have to show that it is possible to complement it with $d-s$ elements in $\overline{S}$ such that the sum of all elements in the resulting set belongs to $X\cup \{\frac{d}{2}\}$. If $\overline{S}$ is the set of all even or all odd elements, $\sum \binom{\overline{S}}{d-s}$ contains all even or all odd elements. Otherwise, Corollary \ref{addcom} implies $|\sum \binom{\overline{S}}{d-s}|>d$, and we are done since $|X\cup \{\frac{d}{2}\}|=d$.

\bf{Case 2: } \rm $2d\notin S'$. In this case we have to show that it is possible to complement it with $d-s$ elements in $\overline{S}$ such that the sum of all elements in the resulting set belongs to $X\cup \{\frac{3d}{2}\}$. This can be done in the same way as we handled Case 1.
\end{proof}

\begin{cla}\label{size1}
For any $S\in \binom{[2d]}{d}$ with $2d\in S$ we have $\binom{S}{1}\cup \binom{S}{d-1}\subseteq \mathcal{F}|_S$.
\end{cla}

\begin{proof}
 There are $4$ types of sets to consider.

\bf{Type 1:} \rm $S'\in \binom{S}{d-1}$ with $2d\in S'$. To prove that $S'$ belongs to $\mathcal{F}|_S$, we have to show that there is a $b\in \overline{S}$ such that $S'\cup \{b\}\in \mathcal{F}_1(X)$. 

Assume that this is not the case. Let $X_1=X\cup \{\frac{d}{2}\}$ and $\overline{S}+\sum_{s\in S'} s\subset Y:=\bar X_1$ (mod $2d$). Given that $|\overline{S}| = |Y| = d$, we have 
$$\overline{S}+\sum_{s\in S'} s = Y.$$ 
Let us put $z:= \sum_{s\in S'} s$ (mod $2d$). Then, since $\overline{S}\sqcup S = X_1\sqcup Y$, we have $S+z = X_1$. Put $X':=\{s'+z:s'\in S'\}$. We get that $z = \sum_{s\in S'} s = \sum_{x\in X'} x-z|X'|$. Rewriting this, we have $\sum_{x\in X'} x = d z,$ and thus $\sum_{x\in X'} x = 0\ ({\rm mod\ } d)$. This contradicts condition 3 from the proposition.

\bf{Type 2: } \rm  $S'\in \binom{S}{d-1}$ with $2d\notin S'$.  To prove that $S'$ belongs to $\mathcal{F}|_S$, we have to show that there is a $b\in \overline{S}$ such that $S'\cup \{b\}\in \mathcal{F}_2(X)$. We can proceed in the same way as in the case of Type 1, with letting $X_1=X\cup \{\frac{3d}{2}\}$.

\bf{Type 3: } \rm  $S'=\{a\}$ with $a\neq2d$. To prove that $S'$ belongs to $\mathcal{F}|_S$, we have to show that there is a $b\in \overline{S}$ such that $\overline{S}\setminus \{b\}\cup \{a\}\in \mathcal{F}_2(X)$.

Assume that this is not the case and let $X_1=X\cup \{\frac{3d}{2}\}$. Then 
$$\Big(a+\sum_{b\in \overline{S}} b\Big) - \overline{S}=\overline{X}_1.$$ 
Let us put $z:= a+\sum_{b\in \overline{S}} b$. Then, since $\overline{S}\sqcup S = X_1\sqcup \overline{X}_1$, we have $z-S = X_1$. Put $x':=z-a$ and note that $x'\in X_1$. 
We get that $z = a+\sum_{b\in \overline{S}} b = z-x'+dz-\sum_{x\in \overline{X}_1} x$. Recall that $\sum_{x\in X_1}x+\sum_{x\in \overline{X}_1} x = d\ ({\rm mod\ } 2d)$. Using the last two equations, we again get $\sum_{x\in X_1\setminus \{x'\}} x = 0\ ({\rm mod\ } d)$. This contradicts condition 3 from the proposition.

\bf{Type 4: } \rm  $S'=\{a\}$ with $a=2d$. To prove that $S'$ belongs to $\mathcal{F}|_S$, we have to show that there is a $b\in \overline{S}$ such that $\overline{S}\setminus \{b\}\cup \{a\}\in \mathcal{F}_1(X)$. We can proceed in the same way as in the case of Type 3, with letting $X_1=X\cup \{\frac{d}{2}\}$.
\end{proof}

\begin{prop}\label{othersets} If $X$ is as in Proposition \ref{almostshattered} and $A\cup \mathcal{F}_1(X)\cup \mathcal{F}_2(X)$ has VC-dimension $(d-1)$ for some $A\subset [2d]$ then $2d\in S$ if $|A|<d$ and $2d\notin A$ if $|A|>d$.
\end{prop}

\begin{proof}We only prove the first half of the statement, the second can be done  similarly. Assume $|A|<d$ and $2d\notin A$. Since any $A'$ of size $d$ that contains $A\cup \{2d\}$ is almost shattered by Proposition \ref{almostshattered}, there is an $S\in \mathcal{F}_1(X)$ such that $S\cap A'=\{2d\}$. This $S$ is disjoint with $A$. By Proposition \ref{almostshattered}, $S$ is almost shattered, and since $S\in \mathcal{F}_1(X)$, $A$ cannot be added without increasing the dimension, as it gives the missing empty projection of $S$.
\end{proof}

Now we are ready to find the constructions even $d=2k$ with $d\geq 6$. It is not hard to check that $$X
=\Big\{2d,\frac d2+1\Big\}\cup\Big[2,\frac{d}{2}-1\Big]\cup \Big[d+1,\frac{3d}{2}-1\Big]$$ 
satisfies the conditions in \mbox{Proposition \ref{almostshattered}.}

Take $\mathcal{F}_1(X)\cup \mathcal{F}_2(X)$, which is a maximal intersecting family in $\binom{[2d]}{d}$, and add some other sets to it, until we obtain a saturated family $\mathcal{F}$. Then it follows from Proposition \ref{othersets} that for any $F\in \mathcal{F}$ the set $F\Delta \{2d\}$ is not contained in $\mathcal{F}$. Thus we can duplicate $\{2d\}$, and obtain a construction on any ground set $[n]$ for $n\geq 2d$.

\subsection{$d=4,5$}\label{sec33}
To complete the proof of Theorem \ref{thmmain}, we need to handle the cases $d=4,5$. We provide two constructions which we have found using computer search.

Let $\sigma$ be the cyclic permutation $(1\ 2 \ldots 2d-1)$ and let $\mathcal F \subset \binom{[2d]}{d}$ be a family. Then we put $\mathcal P(\mathcal F) := \{\sigma^i F, \ F\in\mathcal F, i=1,\ldots, 2d-1\}$.
For $d = 4$ we take $$\mathcal F_4 = \big\{\{1, 2, 3, 4\}, \{1, 2, 3, 5\}, \{1, 2, 4, 5\}, \{1, 2, 4, 8\}, \{1, 3, 5, 8\}\big\},$$ and for $d = 5$ we take 
\begin{align*}
\mathcal F_5 = \big\{ 
& \{1, 2, 3, 4, 5\},
\{1, 2, 3, 4, 6\},
\{1, 2, 3, 4, 8\},
\{1, 2, 3, 5, 6\},
\{1, 2, 3, 5, 7\}, \\
& \{1, 2, 3, 5, 8\},
\{1, 2, 3, 6, 8\},
\{1, 2, 3, 6, 10\},
\{1, 2, 4, 5, 8\},
\{1, 2, 4, 5, 10\}, \\
& \{1, 2, 5, 6, 10\},
\{1, 2, 5, 7, 10\},
\{1, 2, 5, 8, 10\},
\{1, 3, 5, 7, 10\}
\big\}.
\end{align*}
We then check using computer that $\mathcal P(\mathcal F_4), \mathcal P(\mathcal F_5)$ are saturated.

Note that $\mathcal P(\mathcal F_4)$, $\mathcal P(\mathcal F_5)$ are intersecting,
and if adding a set $A\in 2^{[2d]}$ to the family increases the VC-dimension, then adding $\sigma^i A$ increases it as well.
We also note that $\mathcal P(\mathcal F_4)$, $\mathcal P(\mathcal F_5)$ do not have the almost-shattering property.

\section{Acknowledgements} A part of this work was done during the workshop ``Open problems in Combinatorics and Geometry'' in Adygea in October of 2019. We thank Ilya Bogdanov and Yelena Yuditsky for inspiring discussions on the subject. 

The authors acknowledge the financial support from the Ministry of Education and Science of the Russian Federation in the framework of MegaGrant no 075-15-2019-1926. Research of Frankl was partially supported by the National Research, Development, and Innovation Office, NKFIH Grant K119670. Research of Patk\'os was partially supported by the National Research, Development and Innovation Office – NKFIH under the grants SNN 129364 and FK 132060. 
\bibliographystyle{plain}
\bibliography{main}

\begin{thebibliography}{10}

\bibitem{Alo83}
Noga Alon.
\newblock On the density of sets of vectors.
\newblock {\em Discrete Mathematics}, 46(2):199--202, 1983.

\bibitem{AloMorYeh16}
Noga Alon, Shay Moran, and Amir Yehudayoff.
\newblock Sign rank versus {VC} dimension.
\newblock In {\em Conference on Learning Theory}, pages 47--80, 2016.

\bibitem{Baj18}
Béla Bajnok.
\newblock {\em Additive combinatorics: A menu of research problems}.
\newblock CRC Press, 2018.

\bibitem{BalMesWag18}
J{\'o}zsef Balogh, Tam{\'a}s M{\'e}sz{\'a}ros, and Adam~Zsolt Wagner.
\newblock Two results about the hypercube.
\newblock {\em Discrete Applied Mathematics}, 247:322--326, 2018.

\bibitem{BluEhrHau89}
Anselm Blumer, Andrzej Ehrenfeucht, David Haussler, and Manfred~K Warmuth.
\newblock Learnability and the vapnik-chervonenkis dimension.
\newblock {\em Journal of the ACM}, 36:929--965, 1989.

\bibitem{BolLeaRad89}
B{\'e}la Bollob{\'a}s, Imre Leader, and Andrew~J Radcliffe.
\newblock Reverse {K}leitman inequalities.
\newblock {\em Proceedings of the London Mathematical Society}, 3(1):153--168,
  1989.

\bibitem{BolRad95}
B{\'e}la Bollob{\'a}s and Andrew~J Radcliffe.
\newblock Defect {S}auer results.
\newblock {\em Journal of Combinatorial Theory, Series A}, 72:189--208, 1995.

\bibitem{Dud85}
Richard~M Dudley.
\newblock The structure of some {V}apnik-{C}hervonenkis classes.
\newblock In {\em Proceedings of the Berkeley Conference in Honor of Jerzy
  Neyman}, volume~2, pages 495--507, 1985.

\bibitem{Dudl99}
Richard~M Dudley.
\newblock {\em {V}apnik-{Č}ervonenkis Combinatorics}.
\newblock Cambridge Studies in Advanced Mathematics. Cambridge University
  Press, 1999.

\bibitem{ErdHei}
Paul Erd\H{o}s and Hans Heilbronn.
\newblock On the addition of residue classes mod p.
\newblock {\em Acta Arithmetica}, 9:149--159, 1964.

\bibitem{Fra83}
Peter Frankl.
\newblock On the trace of finite sets.
\newblock {\em Journal of Combinatorial Theory, Series A}, 34:41--45, 1983.

\bibitem{Fra89}
Peter Frankl.
\newblock Traces of antichains.
\newblock {\em Graphs and Combinatorics}, 5:295--299, 1989.

\bibitem{FraPac83}
Peter Frankl and J{\'a}nos Pach.
\newblock On the number of sets in a null t-design.
\newblock {\em European Journal of Combinatorics}, 4(1):21--23, 1983.

\bibitem{FurPac91}
Zolt{\'a}n F{\"u}redi and J{\'a}nos Pach.
\newblock Traces of finite sets: extremal problems and geometric applications.
\newblock In {\em Extremal problems for finite sets}, volume~3, pages 255--282.
  J{\'a}nos Bolyai Math. Soc., 1994.

\bibitem{FurQui84}
Zolt{\'a}n F{\"u}redi and Forrest Quinn.
\newblock Traces of finite sets.
\newblock {\em Ars Combinatoria}, 18:195--200, 1984.

\bibitem{GerPat18}
D{\'a}niel Gerbner and Bal{\'a}zs Patk{\'o}s.
\newblock {\em Extremal finite set theory}.
\newblock CRC Press, 2018.

\bibitem{GerGriHam12}
Benjamin Girard, Simon Griffiths, and Yahya~Ould Hamidoune.
\newblock k-sums in abelian groups.
\newblock {\em Combinatorics, Probability and Computing}, 21(4):582--596, 2012.

\bibitem{KreNisRon99}
Ilan Kremer, Noam Nisan, and Dana Ron.
\newblock On randomized one-round communication complexity.
\newblock {\em Computational Complexity}, 8:21--49, 1999.

\bibitem{Mat13}
Ji\v{r}i Matousek.
\newblock {\em Lectures on discrete geometry}, volume 212.
\newblock Springer Science \& Business Media, 2013.

\bibitem{MesRon13}
Tam{\'a}s M{\'e}sz{\'a}ros and Lajos R{\'o}nyai.
\newblock Shattering-extremal set systems of small {VC}-dimension.
\newblock {\em ISRN Combinatorics}, 2013.

\bibitem{Paj85}
Alain Pajor.
\newblock Sous-espaces $\ell^n_1$ des espaces de {B}anach.
\newblock {\em Hermann, Paris, Collection Travaux en cours}, 1985.

\bibitem{sau72}
Norbert Sauer.
\newblock On the density of families of sets.
\newblock {\em Journal of Combinatorial Theory, Series A}, 13:145--147, 1972.

\bibitem{She72}
Saharon Shelah.
\newblock A combinatorial problem; stability and order for models and theories
  in infinitary languages.
\newblock {\em Pacific Journal of Mathematics}, 41:247--261, 1972.

\bibitem{VapChe68}
Vladimir~Naumovich Vapnik and Aleksei~Yakovlevich Chervonenkis.
\newblock The uniform convergence of frequencies of the appearance of events to
  their probabilities.
\newblock In {\em Doklady Akademii Nauk}, volume 181, pages 781--783. Russian
  Academy of Sciences, 1968.

\end{thebibliography}

\end{document}